\documentclass[amstex,12pt, amssymb]{article}

\usepackage{mathtext}
\usepackage[cp1251]{inputenc}
\usepackage[T2A]{fontenc}
\usepackage[dvips]{graphicx}
\usepackage{amsmath}
\usepackage{amssymb}
\usepackage{amsxtra}
\usepackage{latexsym}
\usepackage{ifthen}

\newcounter{lemma}[section]

\newcounter{corollary}[section]

\newcounter{remark}[section]

\newcounter{theorem}[section]

\newcounter{proposition}[section]

\newcounter{example}

\numberwithin{equation}{section}

\begin{document}

\markboth{E.~SEVOST'YANOV}{\centerline{On boundary H\"{o}lder
continuity ...}}

\def\cc{\setcounter{equation}{0}
\setcounter{figure}{0}\setcounter{table}{0}}

\overfullrule=0pt


\author{EVGENY SEVOST'YANOV}

\title{
{\bf ON BOUNDARY H\"{O}LDER CONTINUITY OF MAPPINGS WITH THE POLETSKY
CONDITION}}

\date{\today}
\maketitle

\begin{abstract}
The article is devoted to the study of mappings that distort the
modulus of families of paths by the Poletsky inequality type. At
boundary points of a domain, we have obtained the H\"{o}lder
inequality for such mappings, provided that their characteristic has
finite integral averages over infinitesimal balls. In the
manuscript, we have separately considered the cases of
homeomorphisms and mappings with branching. Also, we have separately
considered good boundaries and domains with prime ends.
\end{abstract}


\section{Introduction}

According to the classical theory of quasiconformal mappings, any
quasiconformal (quasiregular) mapping is H\"{o}lder continuous
inside a domain with some exponent. In particular, the following
result holds (see, e.g., \cite[Theorem~3.2]{MRV$_1$}).

\medskip
{\bf Theorem (Martio, Rickman, V\"{a}is\"{a}l\"{a}, 1970)}. {\sl
Suppose that $f$ is bounded and quasiregular in a domain $G\subset
{\Bbb R}^n$ and $F$ is a compact subset of $G.$ Let $K_I(f)$ be the
smallest constant $K$ for which the inequality $J(x, f)\leqslant
K\cdot (l(f^{\,\prime}(x)))^n$ holds for almost all $x\in G,$ $J(x,
f)=\det f^{\,\prime}(x)$ and
$l(f^{\,\prime}(x))=\min\limits_{|h|=1}|f^{\,\prime}(x)h|.$ Then
there is some constant $\lambda_n$ depending only on $n$ such that
the relation
$$|f(x)-f(y)|\leqslant C|x-y|^{\alpha}$$
holds for $x\in F,$ $y\in G,$ where
$\alpha=(K_I(f))^{\frac{1}{1-n}}$ and $C=\lambda_n(d(f, \partial
G))^{\,-\alpha}d(fG).$}

The main goal of this article is to establish a similar result for
more general classes of mappings on the boundary of a domain. At the
moment, not so much has been done in this direction. We should,
first of all, mention the research on the connection between
H\"{o}lder continuity on the boundary and within the domain,
see~\cite{AMN}, \cite{MN}, and \cite{NP$_1$}--\cite{NP$_2$}.
Separately, we will point to the publication~\cite{NP$_1$}, where
the results on global H\"{o}lder continuity of quasiconformal
mappings were obtained. Here, as a rule, either a definition or a
mapped domain is assumed to be the unit ball. Some progress on
H\"{o}lder continuity at boundary points was also made in relatively
recent publications, see~\cite{AM}, \cite{MSS} and~\cite{RSS}. At
the same time, as far as we know, the most general case, when the
mapping has an unbounded characteristic and acts between two domains
of a more general nature, has not been studied before.

Below we consider four cases, in each of which we prove the
H\"{o}lder continuity for some class of maps. These cases are as
follows: 1) homeomorphisms whose definition domain is good; 2)
mappings with branching, the definition domain of which is good; 3)
homeomorphisms whose definition domain has a bad boundary; 4)
mappings with branching, the definition domain of which has a bad
boundary. All four cases use the same analytical condition, namely,
we assume that the characteristic of mappings has finite averages
over infinitesimal balls.  Such conditions have already used in some
of our papers, see, e.g., \cite[section~7.5]{MRSY}, \cite{MSS} and
\cite{RSS}. As for the proofs of the main results related to each of
the cases, they are all similar in their ''analytical part'',
however, they are somewhat different both from the point of view of
the presence/absence of branching, and different geometry of the
domains under consideration.

\medskip
In what follows,
$$
B(x_0, r)=\{x\in {\Bbb R}^n: |x-x_0|<r\}\,,\qquad {\Bbb B}^n=B(0,
1)\,,$$
$$S(x_0,r) = \{ x\,\in\,{\Bbb R}^n : |x-x_0|=r\}\,, {\Bbb S}^{n-1}=S(0,
1)\,,$$
$$\Omega_n=m({\Bbb B}^n)\,, \omega_{n-1}={\mathcal H}^{n-1}({\Bbb S}^{n-1})\,,$$
$m$ is a Lebesgue measure in ${\Bbb R}^n,$ ${\mathcal H}^{n-1}$ is a
$(n-1)$-measured Hausdorff measure,
\begin{equation}\label{eq49***}
A(x_0, r_1,r_2): =\left\{ x\,\in\,{\Bbb R}^n:
r_1<|x-x_0|<r_2\right\}\,,
\end{equation}
and $M(\Gamma)$ denotes the {\it conformal modulus of the family of
paths $\Gamma$} (see~\cite{Va}). Let $Q:{\Bbb R}^n\rightarrow [0,
\infty]$ be a Lebesgue-measurable function equal to zero outside
$D.$ Consider the following concept, see~\cite[section~7.6]{MRSY}.
We say that a mapping $f:D\rightarrow \overline{{\Bbb R}^n}$ is a
{\it ring $Q$-mapping at a point $x_0\in \overline{D},$} $x_0\ne
\infty,$ if the condition
\begin{equation}\label{eq3*!!}
 M(f(\Gamma(S(x_0, r_1),\,S(x_0, r_2),\,D)))\leqslant
\int\limits_{A} Q(x)\cdot \eta^n(|x-x_0|)\ dm(x)\,, \end{equation}
is fulfilled for some $r_0=r(x_0)>0$ and arbitrary $0<r_1<r_2<r_0,$
where $\eta:(r_1,r_2)\rightarrow [0,\infty ]$ is a arbitrary
nonnegative Lebesgue measurable function satisfying the inequality
\begin{equation}\label{eq28*}
\int\limits_{r_1}^{r_2}\eta(r)\ dr\ \geqslant\ 1\,.
\end{equation}
A domain $G$ in ${\Bbb R}^n$ is called a {\it quasiextremal distance
domain} (short. $QED$-{\it domain}), if there is a number
$A_0\geqslant 1,$ such that the inequality
\begin{equation}\label{eq4***}
M(\Gamma(E, F, {\Bbb R}^n))\leqslant A_0\cdot M(\Gamma(E, F, G))
\end{equation}
holds for any continua $E, F\subset G.$

Given sets $A, B\subset{\Bbb R}^n$, we put
$${\rm diam}\,A=\sup\limits_{x, y\in A}|x-y|\,,\quad {\rm dist}\,(A, B)=
\inf\limits_{x\in A, y\in B}|x-y|\,.$$
Sometimes instead of ${\rm diam}\,A$ and ${\rm dist}\,(A, B)$ we
also write $d(A)$ and $d(A, B),$ respectively.

\medskip
Given numbers $A_0>0,$ $R_0>0$ and $\delta>0,$ a domain $D\subset
{\Bbb R}^n,$ $n\geqslant 2,$ a point $x_0\in \partial D,$
$x_0\ne\infty,$ a path connected continuum $A\subset D$ and a
function $Q:D\rightarrow[0, \infty]$ denoted by $\frak{F}^{A_0,
R_0}_{Q, A, \delta}(D, x_0)$ a family of all homeomorphisms
$f:D\rightarrow B(0, R_0)$ satisfying the
relations~(\ref{eq3*!!})--(\ref{eq28*}) at $x_0$ such that ${\rm
diam\,}(f(A))\geqslant\delta$ and the domain $D_f^{\,\prime}=f(D)$
satisfies the condition~(\ref{eq4***}) with $G\mapsto
D_f^{\,\prime}.$ The following statement holds.

\medskip
\begin{theorem}\label{th3} {\sl\,
Assume that, the following conditions hold: 1) there is
$r^{\,\prime}_0=r^{\,\prime}_0(x_0)>0$ such that, the set $B(x_0,
r)\cap D$ is connected for any $0<r<r^{\,\prime}_0;$ 2) there is
$0<C=C(x_0)<\infty,$ such that
\begin{equation}\label{eq1D}
\limsup\limits_{\varepsilon\rightarrow
0}\frac{1}{\Omega_n\cdot\varepsilon^n}\int\limits_{B(x_0,
\varepsilon)\cap D}Q(x)\,dm(x)\leqslant C\,.
\end{equation}
Then there is
$\widetilde{\varepsilon_0}=\widetilde{\varepsilon_0(x_0)}>0$ and a
number $\alpha=\alpha(n, \delta, C, R_0, x_0)>0$ such that the
relation
\begin{equation}\label{eq2.4.3}
|f(x)-f(y)|\leqslant \alpha\cdot\max\{|x-x_0|^{\beta_n},
|y-x_0|^{\beta_n}\}\end{equation}
holds for any $x, y\in B(x_0, \widetilde{\varepsilon(x_0)})\cap D$
and any $f\in\frak{F}^{A_0, R_0}_{Q, A, \delta}(D, x_0),$ where
$\beta_n=\left(\frac{n\log 2}{A_0C2^{n+1}}\right)^{\frac{1}{n-1}}.$
  }
\end{theorem}

\medskip
\begin{remark}\label{rem1}
In particular, condition~1) of Theorem~\ref{th3} is fulfilled if $D$
is a convex domain.
\end{remark}

\medskip
In the case of mappings with branching, some statement similar to
Theorem~\ref{th3} is also holds, see below. Recall that, a mapping
$f:D\rightarrow\overline{{\Bbb R}^n}$ between domains $D\subset{\Bbb
R}^n$ and $D^{\,\prime}\subset\overline{{\Bbb R}^n}$ is called {\it
closed} if $C(f, \partial D)\subset \partial D^{\,\prime},$ where,
as usual, $C(f, \partial D)$ is the cluster set of the mapping $f$
on $\partial D.$

Later, in the extended space $\overline{{{\Bbb R}}^n}={{\Bbb
R}}^n\cup\{\infty\}$ we use the {\it spherical (chordal) metric}
$h(x,y)=|\pi(x)-\pi(y)|,$ where $\pi$ is a stereographic projection
of $\overline{{{\Bbb R}}^n}$ onto the sphere
$S^n(\frac{1}{2}e_{n+1},\frac{1}{2})$ in ${{\Bbb R}}^{n+1},$ namely,
$$h(x,\infty)=\frac{1}{\sqrt{1+{|x|}^2}}\,,$$
\begin{equation}\label{eq3C}
\ \ h(x,y)=\frac{|x-y|}{\sqrt{1+{|x|}^2} \sqrt{1+{|y|}^2}}\,, \ \
x\ne \infty\ne y
\end{equation}
(see, e.g., \cite[Definition~12.1]{Va}).
Further, for the sets $A, B\subset \overline{{\Bbb R}^n}$ we set
\begin{equation}\label{eq5}
h(E):=\sup\limits_{x,y\in E}h(x, y)\,.
\end{equation}

\medskip
Given numbers $A_0>0,$ $R_0>0$ and $\delta>0,$ a domain $D\subset
{\Bbb R}^n,$ $n\geqslant 2,$ a point $x_0\in \partial D,$
$x_0\ne\infty,$ and a given function $Q:D\rightarrow[0, \infty],$ we
denote $\frak{R}^{A_0, R_0}_{Q, \delta}(D, x_0)$ the family of all
open, discrete and closed ring $Q$-mappings $f:D\rightarrow B(0,
R_0)$ at $x_0$ such that the domain $D_f ^{\,\prime}=f(D)$ satisfies
the condition~(\ref{eq4***}) with $G=D_f^{\,\prime},$ and, in
addition, there exists a path connected continuum $K_f\subset
D^{\,\prime}_f$ such that ${\rm diam\,}(K_f)\geqslant \delta$ and
$h(f^{\,-1}(K_f), \partial D )\geqslant \delta>0.$ The following
theorem holds.

\medskip
\begin{theorem}\label{th4} {\sl\,
Assume that, the following conditions hold: 1) there is
$r^{\,\prime}_0=r^{\,\prime}_0(x_0)>0$ such that, the set $B(x_0,
r)\cap D$ is connected for any $0<r<r^{\,\prime}_0;$ 2) there is
$0<C=C(x_0)<\infty,$ such that
\begin{equation}\label{eq1E}
\limsup\limits_{\varepsilon\rightarrow
0}\frac{1}{\Omega_n\cdot\varepsilon^n}\int\limits_{B(x_0,
\varepsilon)\cap D}Q(x)\,dm(x)\leqslant C\,.
\end{equation}
Then there is
$\widetilde{\varepsilon_0}=\widetilde{\varepsilon_0(x_0)}>0$ and a
number $\alpha=\alpha(n, \delta, C, R_0, x_0)>0$ such that the
relation
\begin{equation}\label{eq2}
|f(x)-f(y)|\leqslant \alpha\cdot\max\{|x-x_0|^{\beta_n},
|y-x_0|^{\beta_n}\}\,,\end{equation}
holds for any $x, y\in B(x_0, \widetilde{\varepsilon(x_0)})\cap D$
and all $f\in\frak{R}^{A_0, R_0}_{Q, \delta}(D, x_0),$ where
$\beta_n=\left(\frac{n\log 2}{A_0C2^{n+1}}\right)^{\frac{1}{n-1}}.$
  }
\end{theorem}

\medskip
Theorems~\ref{th3} and \ref{th4} have corresponding analogs for
domains with bad boundaries. The definition of a prime end used
below may be found in~\cite{ISS}, cf.~\cite{KR} and~\cite{Na$_2$}.
Let $\omega$ be an open set in ${\Bbb R}^k$, $k=1,\ldots,n-1$. A
continuous mapping $\sigma\colon\omega\rightarrow{\Bbb R}^n$ is
called a {\it $k$-dimensional surface} in ${\Bbb R}^n$. A {\it
surface} is an arbitrary $(n-1)$-dimensional surface $\sigma$ in
${\Bbb R}^n.$ A surface $\sigma$ is called {\it a Jordan surface},
if $\sigma(x)\ne\sigma(y)$ for $x\ne y$. In the following, we will
use $\sigma$ instead of $\sigma(\omega)\subset {\Bbb R}^n,$
$\overline{\sigma}$ instead of $\overline{\sigma(\omega)}$ and
$\partial\sigma$ instead of
$\overline{\sigma(\omega)}\setminus\sigma(\omega).$ A Jordan surface
$\sigma\colon\omega\rightarrow D$ is called a {\it cut} of $D$, if
$\sigma$ separates $D,$ that is $D\setminus \sigma$ has more than
one component, $\partial\sigma\cap D=\varnothing$ and
$\partial\sigma\cap\partial D\ne\varnothing$.

A sequence of cuts $\sigma_1,\sigma_2,\ldots,\sigma_m,\ldots$ in $D$
is called {\it a chain}, if:

(i) the set $\sigma_{m+1}$ is contained in exactly one component
$d_m$ of the set $D\setminus \sigma_m,$ wherein $\sigma_{m-1}\subset
D\setminus (\sigma_m\cup d_m)$; (ii)
$\bigcap\limits_{m=1}^{\infty}\,d_m=\varnothing.$

Two chains of cuts  $\{\sigma_m\}$ and $\{\sigma_k^{\,\prime}\}$ are
called {\it equivalent}, if for each $m=1,2,\ldots$ the domain $d_m$
contains all the domains $d_k^{\,\prime},$ except for a finite
number, and for each $k=1,2,\ldots$ the domain $d_k^{\,\prime}$ also
contains all domains $d_m,$ except for a finite number.

The {\it end} of the domain $D$ is the class of equivalent chains of
cuts in $D$. Let $K$ be the end of $D$ in ${\Bbb R}^n$, then the set
$$I(K)=\bigcap\limits_{m=1}\limits^{\infty}\overline{d_m}$$
is called {\it the impression of the end} $K$. One may to prove
that, $I(P)\subset \partial D$ (see, e.g.,
\cite[Proposition~1]{KR}).

Following~\cite{Na$_2$}, we say that the end $K$ is {\it a prime
end}, if $K$ contains a chain of cuts $\{\sigma_m\}$ such that
\begin{equation}\label{eqSIMPLE}
\lim\limits_{m\rightarrow\infty}M(\Gamma(C, \sigma_m, D))=0
\end{equation}
for some continuum $C$ in $D.$ In the following, the following
notation is used: the set of prime ends corresponding to the domain
$D,$ is denoted by $E_D,$ and the completion of the domain $D$ by
its prime ends is denoted $\overline{D}_P.$ Consider the following
definition, which goes back to N\"akki~\cite{Na$_2$}, see
also~\cite{KR}. We say that the boundary of the domain $D$ in ${\Bbb
R}^n$ is {\it locally quasiconformal}, if each point $x_0\in\partial
D$ has a neighborhood $U$ in ${\Bbb R}^n$, which can be mapped by a
quasiconformal mapping $\varphi$ onto the unit ball ${\Bbb
B}^n\subset{\Bbb R}^n$ so that $\varphi(\partial D\cap U)$ is the
intersection of ${\Bbb B}^n$ with the coordinate hyperplane.

The sequence of cuts $\sigma_m,$ $m=1,2,\ldots ,$ is called {\it
regular,} if
$\overline{\sigma_m}\cap\overline{\sigma_{m+1}}=\varnothing$ for
$m\in {\Bbb N}$ and, in addition, $d(\sigma_{m})\rightarrow 0$ as
$m\rightarrow\infty.$ If the end $K$ contains at least one regular
chain, then $K$ will be called {\it regular}. We say that a bounded
domain $D$ in ${\Bbb R}^n$ is {\it regular}, if $D$ can be
quasiconformally mapped to a domain with a locally quasiconformal
boundary whose closure is a compact in ${\Bbb R}^n,$ and, besides
that, every prime end in $D$ is regular. Note that space
$\overline{D}_P=D\cup E_D$ is metric, which can be demonstrated as
follows. If $g:D_0\rightarrow D$ is a quasiconformal mapping of a
domain $D_0$ with a locally quasiconformal boundary onto some domain
$D,$ then for $x, y\in \overline{D}_P$ we put:
\begin{equation}\label{eq1A}
\rho(x, y):=|g^{\,-1}(x)-g^{\,-1}(y)|\,,
\end{equation}
where the element $g^{\,-1}(x),$ $x\in E_D,$ is to be understood as
some (single) boundary point of the domain $D_0.$ It is easy to
verify that~$\rho$ in~(\ref{eq1A}) is a metric on $\overline{D}_P,$
and that the topology on $\overline{D}_P,$ defined by such a method,
does not depend on the choice of the map $g$ with the indicated
property. In what follows,
$$B_{\rho}(P_0, \varepsilon)=\{x\in \overline{D}_P: \rho(x, P_0)<\varepsilon \}\,,$$
where $\rho$ is a metric defined in~(\ref{eq1A}). Note that this
notation is explicitly used in the statements of Theorems~\ref{th1}
and~\ref{th2}, see below.

We say that a sequence $x_m\in D,$ $m=1,2,\ldots,$ converges to a
prime end of $P\in E_D$ as $m\rightarrow\infty, $ if for any $k\in
{\Bbb N}$ all elements $x_m$ belong to $d_k$ except for a finite
number. Here $d_k$ denotes a sequence of nested domains
corresponding to the definition of the prime end $P.$ Note that for
a homeomorphism of a domain $D$ onto $D^{\,\prime},$ the end of the
domain $D$ uniquely corresponds to some sequence of nested domains
in the image under the mapping.

\medskip
Given numbers $A_0>0,$ $R_0>0$ and $\delta>0,$ a domain $D\subset
{\Bbb R}^n,$ $n\geqslant 2,$ a point $P_0\in E_D,$ a path connected
continuum $A\subset D$ and a function $Q:D\rightarrow[0, \infty]$
denoted by $\frak{F}^{A_0, R_0}_{Q, A, \delta}(D, P_0)$ a family of
all homeomorphisms $f:D\rightarrow B(0, R_0)$ satisfying the
relations~(\ref{eq3*!!})--(\ref{eq28*}) for any $x_0\in I(P_0)$
(where $I(P_0)$ denotes the impression of $P_0$) such that ${\rm
diam\,}(f(A))\geqslant\delta$ and the domain $D_f^{\,\prime}=f(D)$
satisfies the condition~(\ref{eq4***}) with $G\mapsto
D_f^{\,\prime}.$ The following statement holds.

\medskip
\begin{theorem}\label{th1} {\sl\,
Assume that the domain $D$ is regular and the following conditions
are fulfilled: 1) for each $y_0\in\partial D$ there exists
$r^{\,\prime}_0=r^{\,\prime}_0(y_0) >0$ such that the set $B(x_0,
r)\cap D$ is finitely connected for all $0<r<r^{\,\prime}_0,$ and,
for each component $K$ of the set $B(y_0, r)\cap D$ the following
condition is fulfilled: any $x, y\in K$ may be joined by a path
$\gamma:[a, b]\rightarrow {\Bbb R}^n$ such that $|\gamma|\in K\cap
\overline{B(x_0 , \max\{|x-x_0|, |y-x_0|\})},$
$$|\gamma|=\{x\in {\Bbb R}^n: \exists\,t\in[a, b]: \gamma(t)=x\};$$
2) for each $x_0\in I(P_0)$ there is $0<C=C(x_0)<\infty$ such that
\begin{equation}\label{eq1F}
\limsup\limits_{\varepsilon\rightarrow
0}\frac{1}{\Omega_n\cdot\varepsilon^n}\int\limits_{B(x_0,
\varepsilon)\cap D}Q(x)\,dm(x)\leqslant C\,.
\end{equation}
Then for each $P\in E_D$ there exists $y_0\in \partial D$ such that
$I(P_0)={y_0},$ where $I(P)$ denotes the impression of $P.$ In
addition, there exists $\rho_0=\rho_0(P_0)>0$ and a number
$\alpha=\alpha(n, \delta, C, R_0, P_0)>0$ such that the inequality
\begin{equation}\label{eq3A}
|f(x)-f(y)|\leqslant \alpha\cdot\max\{|x-x_0|^{\beta_n},
|y-x_0|^{\beta_n}\}\,,\end{equation}
holds for all $x, y\in B_{\rho}(P_0, \rho_0)\cap D$ and all
$f\in\frak{F}^{A_0, R_0}_{Q, A, \delta}(D, P_0),$ where
$\{x_0\}=I(P_0)$ and $\beta_n=\left(\frac{n\log
2}{A_0C2^{n+1}}\right)^{\frac{1}{n-1}}.$
  }
\end{theorem}

\medskip
\begin{remark}\label{rem2}
In particular, condition~1) of Theorem~\ref{th1} is fulfilled if,
for any $y_0\in\partial D$ there exists
$r^{\,\prime}_0=r^{\,\prime}_0(y_0) >$0 such that the set $B(y_0,
r)\cap D$ is finitely connected for all $0<r<r^{\,\prime}_0,$ and,
for any component $K$ of $B(y_0, r)\cap D$ is convex. Indeed, let
$x, y\in K.$ Let us join $x$ and $y$ by a segment $\gamma$ inside
$K$ (this is possible because the connected open set $K$ is also
path connected , see, e.g., \cite[Corollary~13.1]{MRSY}). Let us
assume that $|x-y_0|\geqslant |y-y_0|.$ Due to the fact that the
ball $\overline{B(x, |x-y_0|)}$ is convex, the entire segment
$\gamma $ belongs to $B(x, |x-y_0|).$ Then $\gamma$ is the desired
path.
\end{remark}

\medskip
Given numbers $A_0>0,$ $R_0>0$ and $\delta>0,$ a domain $D\subset
{\Bbb R}^n,$ $n\geqslant 2,$ a point $P_0\in E_D,$ and a given
function $Q:D\rightarrow[0, \infty],$ we denote $\frak{R}^{A_0,
R_0}_{Q, \delta}(D, P_0)$ the family of all open, discrete and
closed mappings $f:D\rightarrow B(0, R_0)$ satisfying the
relations~(\ref{eq3*!!})--(\ref{eq28*}) for any $x_0\in I(P_0)$
(where $I(P_0)$ denotes the impression of $P_0$) such that the
domain $D_f ^{\,\prime}=f(D)$ satisfies the condition~(\ref{eq4***})
with $G=D_f^{\,\prime},$ and, in addition, there exists a path
connected continuum $K_f\subset D^{\,\prime}_f$ such that ${\rm
diam\,}(K_f)\geqslant \delta$ and $h(f^{\,-1}(K_f),
\partial D )\geqslant \delta>0.$ The following theorem holds.

\medskip
\begin{theorem}\label{th2} {\sl\,
Assume that the domain $D$ is regular and the following conditions
are fulfilled: 1) for each $y_0\in\partial D$ there exists
$r^{\,\prime}_0=r^{\,\prime}_0(y_0)>0$ such that the set $B(y_0,
r)\cap D$ is finitely connected for all $0<r<r^{\,\prime}_0,$ and,
for each component $K$ of the set $B(y_0, r)\cap D$ the following
condition is fulfilled: any $ x, y\in K$ may be joined by a path
$\gamma:[a, b]\rightarrow {\Bbb R}^n$ such that $|\gamma|\in K\cap
\overline{B(y_0 , \max\{|x-y_0|, |y-y_0|\})},$
$$|\gamma|=\{x\in {\Bbb R}^n: \exists\,t\in[a, b]: \gamma(t)=x\};$$
2) for each $x_0\in I(P_0)$ there is $0<C=C(x_0)<\infty$ such that
\begin{equation}\label{eq1G}
\limsup\limits_{\varepsilon\rightarrow
0}\frac{1}{\Omega_n\cdot\varepsilon^n}\int\limits_{B(x_0,
\varepsilon)\cap D}Q(x)\,dm(x)\leqslant C\,.
\end{equation}
Then, for any $P\in E_D$ there is $y_0\in
\partial D$ such that $I(P)=\{y_0\},$ where $I(P)$ denotes the impression of~$P.$
In addition, there is $\rho_0=\rho_0(P_0)>0$  and a number
$\alpha=\alpha(n, \delta, C, R_0, x_0)>0$ such that the relation
\begin{equation}\label{eq4A}
|f(x)-f(y)|\leqslant \alpha\cdot\max\{|x-x_0|^{\beta_n},
|y-x_0|^{\beta_n}\}\,,\end{equation}
holds for any $x, y\in B_{\rho}(P_0, \varepsilon_0)\cap D$ and all
$\frak{R}^{A_0, R_0}_{Q, \delta}(D, P_0),$ where $\{x_0\}=I(P_0)$
and $\beta_n=\left(\frac{n\log
2}{A_0C2^{n+1}}\right)^{\frac{1}{n-1}}.$
  }
\end{theorem}

\section{Preliminaries}

Let $a>0$ and let $\varphi\colon [a,\infty)\rightarrow[0,\infty)$ be
a nondecreasing function such that, for some constants $\gamma>0,$
$T>0 $ and all $t\geqslant T$, the inequality
\begin{equation}\label{eq1B}
\varphi(2t)\leqslant \gamma\cdot\varphi(t)
\end{equation}
is fulfilled. We will call such functions {\it functions that
satisfy the doubling condition}.

Let $\varphi\colon [a,\infty)\rightarrow[0,\infty)$ be a function
with the doubling condition, then the function
$\widetilde{\varphi}(t):=\varphi(1/t)$ does not increase and is
defined on a half-interval $(0, 1/a].$ The following statement is
proved in~\cite[Lemma~3.1]{RSS}.

\medskip
\begin{proposition}\label{pr1}
{\sl Let $a>0,$ let $\varphi\colon [a,\infty)\rightarrow[0,\infty)$
be a nondecreasing function with a doubling condition~(\ref{eq1B}),
let $x_0\in {\Bbb R}^n,$ $n\geqslant 2,$ and let $Q:{\Bbb
R}^n\rightarrow [0, \infty]$ be a Lebesgue measurable function for
which there exists $0<C<\infty$ such that
\begin{equation}\label{eq1AA}
\limsup\limits_{\varepsilon\rightarrow
0}\frac{\varphi(1/\varepsilon)}{\Omega_n\cdot\varepsilon^n
}\int\limits_{B(x_0, \varepsilon)}Q(x)\,dm(x)\leqslant C\,.
\end{equation}
Then there exists $\varepsilon^{\,\prime}_0>0$ such that
\begin{equation}\label{eq2B}
\int\limits_{\varepsilon<|x-x_0|<\varepsilon^{\,\prime}_0}
\frac{\varphi(1/|x-x_0|)Q(x)\,dm(x)}{|x-x_0|^n}\leqslant
C_1\cdot\left(\log\frac{1}{\varepsilon}\right)\,,\qquad\varepsilon\rightarrow
0\,,
\end{equation}
where $C_1:=\frac{\gamma C\Omega_n2^n}{\log 2}.$
}
\end{proposition}

A domain $R$ in $\overline{{\Bbb R}^n},$ $n\geqslant 2,$ is called
{\it a ring,} if $\overline{{\Bbb R}^n}\setminus R$ consists of
exactly two components $E$ and $F.$ In this case, we write: $R=R(E,
F).$ The following statement is true, see \cite[ratio~(7.29)]{MRSY}.

\medskip
\begin{proposition}\label{pr3}{\sl\, If $R=R(E, F)$ is a ring, then
$$M(\Gamma(E, F, \overline{{\Bbb R}^n}))\geqslant\frac{\omega_{n-1}}{\left(
\log\frac{2\lambda^2_n}{h(E)h(F)}\right)^{n-1}}\,,$$
where $\lambda_n \in[4,2e^{n-1}),$ $\lambda_2=4$ and
$\lambda_n^{1/n} \rightarrow e$ as $n\rightarrow \infty,$ and $h(E)$
denotes the chordal diameter of the set $E$ defined in~(\ref{eq5}).
}
\end{proposition}

\medskip
Let $D\subset {\Bbb R}^n,$ $f:D\rightarrow {\Bbb R}^n$ be an open
discrete mapping, let $\beta: [a,\,b)\rightarrow {\Bbb R}^n$ be a
path and let $x\in\,f^{\,-1}(\beta(a)).$ A path $\alpha:
[a,\,c)\rightarrow D$ is called a {\it maximal $f$-lifting} of
$\beta$ with the origin at the point $x,$ if $(1)\quad
\alpha(a)=x\,;$ $(2)\quad f\circ\alpha=\beta|_{[a,\,c)};$ $(3)$\quad
for every $c<c^{\prime}\leqslant b$, there is no a path
$\alpha^{\prime}: [a,\,c^{\prime})\rightarrow D$ such that
$\alpha=\alpha^{\prime}|_{[a,\,c)}$ and $f\circ
\alpha^{\,\prime}=\beta|_{[a,\,c^ {\prime})}.$ The following
statement is true, see~\cite[lemma~3.12]{MRV$_2$},
cf.~\cite[lemma~3.7]{Vu}.

\medskip
\begin{proposition}\label{pr2}
{\sl Let $f:D\rightarrow {\Bbb R}^n,$ $n\geqslant 2,$ be open
discrete mapping, let $x_0\in D,$ and let $\beta:[a,\,b)\rightarrow
{\Bbb R}^n$ be a path such that $\beta(a)=f(x_0 )$ and either
$\lim\limits_{t\rightarrow b}\beta(t)$ exists, or
$\beta(t)\rightarrow \partial f(D)$ as $t\rightarrow b.$ Then
$\beta$ has a maximal $f$-lifting of $\alpha: [a,\,c)\rightarrow D $
starting at the point $x_0.$ If $\alpha(t)\rightarrow x_1\in D$ as
$t\rightarrow c,$ then $c=b$ and $f(x_1)=\lim\limits_{t \rightarrow
b}\beta(t).$ Otherwise, $\alpha(t)\rightarrow \partial D$ as
$t\rightarrow c.$}
\end{proposition}

\section{Proof of Theorem~\ref{th3}}

Let
$$\varepsilon_0=\min\{\varepsilon^{\,\prime}_0(x_0),
r_0^{\,\prime}, {\rm dist}\,(x_0, A), 1\}\,,$$
where $\varepsilon^{\,\prime}_0>0$ is a number from
Proposition~\ref{pr1} with $\varphi\equiv 1,$ besides that
$r_0^{\,\prime}$ is a number from conditions of the theorem, and $A$
is a continuum from the definition of $\frak{F}^{A_0, R_0}_{Q, A,
\delta}(D, x_0).$ Let $x, y\in B(x_0, \varepsilon^2_0)$ and let
$f\in \frak{F}^{A_0, R_0}_{Q, A, \delta}(D, x_0).$ Now, $x, y\in
B(x_0, \varepsilon_0).$ Without loss of a generalization, we may
assume that $|x-x_0|\geqslant|y-x_0|.$
By the definition of $\varepsilon_0,$
\begin{equation}\label{eq18}
A\subset D\setminus B(x_0, \varepsilon_0)\,.
\end{equation}
Since the points $S(x_0, r)\cap D,$ $0<r<r^{\,\prime}_0,$ are
accessible from the domain $D$ by means of some path $\gamma,$ and
the set $B(x_0, r)\cap D$ is connected for all $0<r<r^{\,\prime}_0,$
we may join points $x$ and $y$ by a path $K,$ which completely
belongs to the ball $\overline{B(x_0, |x-x_0|)}$ and belongs to $D.$
We may consider that, $K$ is a Jordan path. Let $z, w\in f(A)\subset
D^{\,\prime}_{f}$ and $u, v\in A$ be such that
\begin{equation}\label{eq8}
{\rm diam}\,f(A)=|z-w|=|f(u)-f(v)|\geqslant \delta\,.
\end{equation}
Since the continuum $A$ is path connected, one can join the points
$u, v$ by a path $K^{\,\prime}$ inside $A.$ Due to~(\ref{eq8}), we
obtain that
\begin{equation}\label{eq9}
{\rm diam}\,f(|K^{\,\prime}|)\geqslant |f(u)-f(v)|\geqslant
\delta\,.
\end{equation}
We also may consider that $K^{\,\prime}$ is Jordan.

Note that, $f(K)$ and $f(K^{\,\prime})$ are also Jordan paths, and
they do not split ${\Bbb R}^n.$ Indeed, for $n\geqslant 3$ the set
$f(|K|)\cup f(|K^{\,\prime}|)$ has a topological dimension~1 as the
union of two closed sets of topological dimension~1
(see~\cite[Theorem~III 2.3]{HW}). Then $f(|K|)\cup
f(|K^{\,\prime}|)$ does not split ${\Bbb R}^n$
(see~\cite[Corollary~1.5.IV]{HW}). Now, let $n=2.$ According to
Antoine's theorem on the absence of wild arcs
(see~\cite[Theorem~II.4.3]{Keld}), there exists a homeomorphism
$\varphi:{\Bbb R}^2\rightarrow {\Bbb R}^2,$ which maps $f(|K|)$ onto
some segment $I.$ It follows that, any points $x,y\in {\Bbb
R}^2\setminus f(|K|)$ may be joined by a path $\gamma$ in ${\Bbb
R}^2\setminus f(|K|).$ Reasoning similarly, it can be shown that,
any points $x,y\in {\Bbb R}^2\setminus (f(|K|)\bigcup
f(|K^{\,\prime}|))$ may be joined by a path $\gamma$ in ${\Bbb
R}^2\setminus (f(|K|)\bigcup f(|K^{\,\prime}|)).$

Therefore, $R=R(f(|K|), f(|K^{\,\prime}|))$ is a ring domain. In
this case, let us to set $\Gamma=\Gamma(f(|K|)), f(|K^{\,\prime}|),
{\Bbb R}^n).$ Then, by Proposition~\ref{pr3}
\begin{equation}\label{eq10}
M(\Gamma(f(|K|), f(|K^{\,\prime}|), \overline{{\Bbb
R}^n}))\geqslant\frac{\omega_{n-1}}{\left(
\log\frac{2\lambda^2_n}{h(f(|K|))h(f(|K^{\,\prime}|))}\right)^{n-1}}\,.
\end{equation}
By~(\ref{eq9}) and~(\ref{eq10}), and by the definition of~$K,$ we
obtain that
\begin{equation}\label{eq15}
M(\Gamma(f(|K|), f(|K^{\,\prime}|), \overline{{\Bbb
R}^n}))\geqslant\frac{\omega_{n-1}}{\left(
\log\frac{2\lambda^2_n}{\delta\cdot h(f(x), f(y))}\right)^{n-1}}\,.
\end{equation}
Since $D^{\,\prime}_{f}$ is a $QED$-domain with a constant $A_0$
in~(\ref{eq4***}), by the condition~(\ref{eq15}) we obtain that
\begin{equation}\label{eq16}
M(\Gamma(f(|K|), f(|K^{\,\prime}|),
D^{\,\prime}_{f}))\geqslant\frac{\omega_{n-1}}{A_0\left(
\log\frac{2\lambda^2_n}{\delta\cdot h(f(x), f(y))}\right)^{n-1}}\,.
\end{equation}
Note that
\begin{equation}\label{eq17}
\Gamma(|K|, |K^{\,\prime}|, D)>\Gamma(S(x_0, |x-x_0|), S(x_0,
\varepsilon_0), D)\,.
\end{equation}
Indeed, let $\gamma\in \Gamma(|K|, |K^{\,\prime}|, D),$ $\gamma:[0,
1]\rightarrow D,$ $\gamma(0)\in |K|$ and $\gamma(1)\in
|K^{\,\prime}|.$ Since $ |K|\subset B(x_0, |x-x_0|),$ and
$|K^{\,\prime}|\subset A,$ due to~(\ref{eq18}), we obtain that
$A\subset D\setminus B(x_0, |x-x_0|).$ Then $\gamma\cap B (x_0,
|x-x_0|)\ne\varnothing\ne \gamma\cap (D\setminus B(x_0, |x-x_0|)).$
By~\cite[Theorem~1.I.5.46]{Ku} there exists $t_1\in (0, 1)$ such
that $\gamma(t_1)\in S(x_0, |x-x_0|).$ Consider a path
$\gamma_1:=\gamma|_{[t_1, 1]}.$ Since $|K|\subset B(x_0,
\varepsilon_0)$ and $|K^{\,\prime}|\subset A,$ due to~(\ref{eq18}),
$A\subset D\setminus B(x_0, \varepsilon_0).$ Then $\gamma_1\cap
B(x_0, \varepsilon_0)\ne\varnothing\ne \gamma\cap (D\setminus B(x_0,
\varepsilon_0)).$ By~\cite[Theorem~1.I.5.46]{Ku}, there exists
$t_2\in (0, t_1)$ such that $\gamma_1(t_1)=\gamma(t_1)\in S (x_0,
\varepsilon_0).$ Put $\gamma_2:=\gamma|_{[t_1, t_2]}.$ Then
$\gamma_2$ is a subpath of $\gamma$ and $\gamma_2\in\Gamma(S(x_0
,|x-x_0|), S(x_0, \varepsilon_0), D).$ This proves~(\ref{eq17}). In
this case, by the minorization of a modulus of families of paths
(see, e.g., \cite[Theorem~6.4]{Va}), by~(\ref{eq17})
and~(\ref{eq3*!!}) we obtain that
$$M(\Gamma(f(|K|), f(|K^{\,\prime}|), D^{\,\prime}_{f}))=
M(f(\Gamma(|K|, |K^{\,\prime}|, D)))\leqslant $$
\begin{equation}\label{eq20}
\leqslant M(f(\Gamma(S(x_0, |x-x_0|), S(x_0, \varepsilon_0),
D)))\leqslant\int\limits_{A} Q(x)\cdot \eta^n(|x-x_0|)\, dm(x)\,,
\end{equation}
where $\eta$ is an arbitrary Lebesgue measurable function satisfying
the condition~(\ref{eq28*}) for $r_1=|x-x_0|,$ $r_2=\varepsilon_0.$
Set
$$\eta(t):=\begin{cases}\frac{1}
{\left(\log\frac{\varepsilon_0}{|x-x_0|}\right)
t}\,,& t\in(|x-x_0|, \varepsilon_0)\,, \\
0\,,& t\not\in(|x-x_0|, \varepsilon_0)\,.
\end{cases}$$
Observe that, the function $\eta$ satisfies the
condition~(\ref{eq28*}) for $r_1=|x-x_0|$ and $r_2=\varepsilon_0.$
Now, by~(\ref{eq20}) we obtain that
\begin{equation}\label{eq21}
M(\Gamma(f(|K|), f(|K^{\,\prime}|), D^{\,\prime}_{f}))\leqslant
\frac{1}{\log^n\frac{\varepsilon_0}{|x-x_0|}}\int\limits_{A(x_0,
|x-x_0|, \varepsilon_0)} \frac{Q(x)}{{|x-x_0|}^n} \,dm(x)\,.
\end{equation}
Since $|x-x_0|<\varepsilon^2_0,$
\begin{equation}\label{eq3}
\log\frac{1}{|x-x_0|}<2\log\frac{\varepsilon_0}{|x-x_0|}\,.
\end{equation}
By the choice of $\varepsilon_0,$ by Proposition~\ref{pr1} and
by~(\ref{eq21}) we obtain that
\begin{equation}\label{eq22}
M(\Gamma(f(|K|), f(|K^{\,\prime}|), D^{\,\prime}_{f}))\leqslant
\frac{2}{\log^{n-1}\frac{\varepsilon_0}{|x-x_0|}}\frac{C\Omega_n2^n}{\log
2}\,.
\end{equation}
Combining~(\ref{eq16}) and~(\ref{eq22}), we obtain that
\begin{equation}\label{eq23}
\frac{\omega_{n-1}}{A_0\left( \log\frac{2\lambda^2_n}{\delta\cdot
h(f(x), f(y))}\right)^{n-1}}\leqslant
\frac{2}{\log^{n-1}\frac{\varepsilon_0}{|x-x_0|}}\frac{C\Omega_n2^n}{\log
2}\,.
\end{equation}
Since $\omega_{n-1}=n\Omega_n,$ the latter relation may be rewritten
as
\begin{equation}\label{eq24}
\frac{1}{\log^{n-1}\frac{2\lambda^2_n}{\delta\cdot h(f(x),
f(y))}}\leqslant
\frac{1}{\log^{n-1}\frac{\varepsilon_0}{|x-x_0|}}\frac{A_0C2^{n+1}}{n\log
2}\,.
\end{equation}
By~(\ref{eq24}) we obtain that
$$
\log\frac{2\lambda^2_n}{\delta\cdot h(f(x), f(y))}\geqslant
\log\frac{\varepsilon_0}{|x-x_0|}\cdot \left(\frac{n\log
2}{A_0C2^{n+1}}\right)^{\frac{1}{n-1}}\,,
$$
$$
\frac{2\lambda^2_n}{\delta\cdot h(f(x), f(y))}\geqslant
\left(\frac{\varepsilon_0}{|x-x_0|}\right)^{ \left(\frac{n\log
2}{A_0C2^{n+1}}\right)^{\frac{1}{n-1}}}\,,
$$
\begin{equation}\label{eq25}
h(f(x), f(y))\leqslant \frac{2\lambda^2_n}{\delta
{\varepsilon_0}^{\left(\frac{n\log
2}{A_0C2^{n+1}}\right)^{\frac{1}{n-1}}}}|x-x_0|^{ \left(\frac{n\log
2}{A_0C2^{n+1}}\right)^{\frac{1}{n-1}}}\,.\end{equation}
By the definition of the chordal distance in~(\ref{eq3C}) and due to
the condition~$f(D)=D^{\,\prime}_{f}\subset B(0, R_0),$ we obtain
that
$$h(f(x),f(y))\geqslant \frac{|f(x)-f(y)|}{1+R^2_0}\,.$$
Then, by~(\ref{eq25}), we have that
\begin{equation}\label{eq26}
|f(x)-f(y)|\leqslant \frac{2(1+R^2_0)\lambda^2_n}{\delta
{\varepsilon_0}^{\left(\frac{n\log
2}{A_0C2^{n+1}}\right)^{\frac{1}{n-1}}}}|x-x_0|^{ \left(\frac{n\log
2}{A_0C2^{n+1}}\right)^{\frac{1}{n-1}}}\,,\end{equation}
or
\begin{equation}\label{eq27}
|f(x)-f(y)|\leqslant \frac{2(1+R^2_0)\lambda^2_n}{\delta
{\varepsilon_0}^{\beta_n}}|x-x_0|^{\beta_n}\,,
\end{equation}
because~$\beta_n=\left(\frac{n\log
2}{A_0C2^{n+1}}\right)^{\frac{1}{n-1}}.$ We set now
$\widetilde{\varepsilon_0}:=\varepsilon^2_0.$ The
relation~(\ref{eq27}) completes the proof.~$\Box$

\medskip
\begin{corollary}\label{cor1}
{\sl\, Under the conditions of Theorem~\ref{th3}, any
$f\in\frak{F}^{A_0, R_0}_{Q, A, \delta}(D, x_0)$ has a continuous
extension at~$x_0$ and, in addition,
\begin{equation}\label{eq24B}
 |f(x)-f(x_0)|\leqslant \alpha\cdot|x-x_0|^{\beta_n}\,.
\end{equation}
  }
\end{corollary}
\begin{proof}
If the mapping $f$ did not have a limit as $x\rightarrow x_0,$ then
we would construct at least two sequences $x_m\rightarrow x_0$ and
$y_m\rightarrow x_0,$ $m\rightarrow\infty,$ such that
$|f(x_m)-f(y_m)|\geqslant \delta>0$ for some positive $\delta>0$ and
all $m=1,2,\ldots .$ But this contradicts the
inequality~(\ref{eq2.4.3}). Therefore, the limit of $f$ as
$x\rightarrow x_0$ exists. To prove the inequality~(\ref{eq24B}), it
remains to pass in the relation~(\ref{eq2.4.3}) to the limit as
$y\rightarrow x_0.$~$\Box$
\end{proof}

\section{Proof of Theorem~\ref{th4}}

Let
$$\varepsilon_0=\min\{\varepsilon^{\,\prime}_0(x_0),
r_0^{\,\prime}, \delta, 1\}\,,$$
where $\varepsilon^{\,\prime}_0>0$ is a number from
Proposition~\ref{pr1} with $\varphi\equiv 1,$ besides that
$r_0^{\,\prime}$ is a number from conditions of the theorem, and
$\delta$ is a number from the definition of $\frak{R}^{A_0, R_0}_{Q,
\delta}(D, x_0).$ We set
$\widetilde{\varepsilon_0}:=\varepsilon^2_0.$ Let $x, y\in B(x_0,
\widetilde{\varepsilon_0})$ and let $f\in \frak{R}^{A_0, R_0}_{Q,
\delta}(D, x_0).$ Now, $x, y\in B(x_0, \varepsilon_0).$ Without loss
of a generalization, we may assume that $|x-x_0|\geqslant|y-x_0|.$

Let $K_{f}\subset D^{\,\prime}_{f}$ be a path connected continuum
such that ${\rm diam\,} (K_{f})\geqslant \delta$ and, in addition,
$h(f^{\,-1}(K_{f}),
\partial D)\geqslant \delta>0$ (such a continuum exists by the definition of
the class $\frak{R}^{A_0, R_0}_{Q, \delta}(D, x_0)$). By the
definition of $\varepsilon_0,$
\begin{equation}\label{eq18A}
f^{\,-1}(K_{f})\subset D\setminus B(x_0, \varepsilon_0)\,.
\end{equation}
Since the points $S(x_0, r)\cap D,$ $0<r<r^{\,\prime}_0,$ are
accessible from the domain $D$ by means of some path $\gamma,$ and
the set $B(x_0, r)\cap D$ is connected for all $0<r<r^{\,\prime}_0,$
the points $x$ and $y$ may be joined by a path $K,$ which belongs
entirely to the ball $\overline{B(x_0, |x-x_0|)}$ and belongs
to~$D.$ We may consider that $K$ is a Jordan path. Let $z, w\in
K_{f}\subset D^{\,\prime}_{f}$ be such that
\begin{equation}\label{eq8A}
{\rm diam}\,K_{f}=|z-w|\geqslant \delta\,.
\end{equation}
Since $K_{f}$ is path connected, we may join points $z, w$ by a
Jordan path~$K^{\,\prime}$ inside $K_{f}.$ Due to~(\ref{eq8A}) we
obtain that
\begin{equation}\label{eq9A}
{\rm diam}\,|K^{\,\prime}|\geqslant |z-w|\geqslant \delta\,.
\end{equation}
If the path $f(|K|))$ is not Jordan, we discard from $f(|K|))$ no
more than a finite number of its loops. Let $|D|\subset f(|K|))$ be
a locus of the Jordan path $D,$ which is obtained by such rejection.
Just as in the proof of Theorem~\ref{th3}, it may be shown that
$R=R(|D|, |K^{\,\prime}|)$ is a ring in ${\Bbb R}^n.$

In this case, we denote $\Gamma=\Gamma(|D|, |K^{\,\prime}|, {\Bbb
R}^n).$ Now, by Proposition~\ref{pr3}
\begin{equation}\label{eq10A}
M(\Gamma(|D|, |K^{\,\prime}|, {\Bbb
R}^n))\geqslant\frac{\omega_{n-1}}{\left(
\log\frac{2\lambda^2_n}{h(|K^{\,\prime}|)h(|D|)}\right)^{n-1}}\,.
\end{equation}
Due to~(\ref{eq9A}) and~(\ref{eq10A}), and by the definition of the
path~$K^{\,\prime},$ we obtain that
\begin{equation}\label{eq15A}
M(\Gamma(|D|, |K^{\,\prime}|, {\Bbb
R}^n))\geqslant\frac{\omega_{n-1}}{\left(
\log\frac{2\lambda^2_n}{\delta\cdot h(f(x), f(y))}\right)^{n-1}}\,.
\end{equation}
Since~$D^{\,\prime}_{f}$ is a $QED$-domain with a constant $A_0$
in~(\ref{eq4***}), by~(\ref{eq15A}) we obtain that
\begin{equation}\label{eq16A}
M(\Gamma(|D|, |K^{\,\prime}|,
D^{\,\prime}_{f}))\geqslant\frac{\omega_{n-1}}{A_0\left(
\log\frac{2\lambda^2_n}{\delta\cdot h(f(x), f(y))}\right)^{n-1}}\,.
\end{equation}
Let $\Gamma^{\,*}$ be a family $\gamma:[0, 1)\rightarrow D$ of
maximal $f$-liftings of paths $\gamma^{\,\prime}:[0, 1]\rightarrow
D^{\,\prime}_{f}$ from $\Gamma=\Gamma(|D|, |K^{\,\prime}|,
D^{\,\prime}_{f})$ starting at $|K|.$ Such liftings exist by
Proposition~\ref{pr2}. By the same Proposition, due to the closeness
of~$f,$ we obtain that each path $\gamma\in \Gamma^{\,*}$ has an
extension $\gamma:[0, 1]\rightarrow D$ to the point $b=1.$ Then
$\gamma(1)\in f^{\,-1}(K_{f}),$ that is, $\Gamma^{\,*}\subset
\Gamma(|K|, f^{\,-1}(K_{f}), D).$

Reasoning similar to the proof of Theorem~\ref{th3}, we may show
that
\begin{equation}\label{eq17A}
\Gamma(|K|, f^{\,-1}(K_{f}), D)>\Gamma(S(x_0, |x-x_0|), S(x_0,
\varepsilon_0), D)\,.
\end{equation}
Observe that, $f(\Gamma^{\,*})=\Gamma=\Gamma(|D|, |K^{\,\prime}|,
D^{\,\prime}_{f}).$ In this case, by the minorization of the modulus
(see., e.g., \cite[Theorem~6.4]{Va}), by~(\ref{eq17A})
and~(\ref{eq3*!!}) we obtain that
$$M(f(\Gamma^{\,*}))=M(\Gamma(|D|,
|K^{\,\prime}|, D^{\,\prime}_{f}))\leqslant $$
\begin{equation}\label{eq20A}
\leqslant M(f(\Gamma(S(x_0, |x-x_0|), S(x_0, \varepsilon_0),
D)))\leqslant\int\limits_{A} Q(x)\cdot \eta^n(|x-x_0|)\ dm(x)\,,
\end{equation}
where $\eta$ is an arbitrary nonnegative Lebesgue measurable
function satisfying the relation~(\ref{eq28*}) for $r_1=|x-x_0|,$
$r_2=\varepsilon_0.$ Put
$$\eta(t):=\begin{cases}\frac{1}
{\left(\log\frac{\varepsilon_0}{|x-x_0|}\right)
t}\,,& t\in(|x-x_0|, \varepsilon_0)\,, \\
0\,,& t\not\in(|x-x_0|, \varepsilon_0)\,.
\end{cases}$$
Observe that, the function $\eta$ satisfies the
condition~(\ref{eq28*}) for $r_1=|x-x_0|,$ $r_2=\varepsilon_0.$ Now,
by the relation~(\ref{eq20}) we obtain that
\begin{equation}\label{eq21A}
M(\Gamma(|D|, |K^{\,\prime}|, D^{\,\prime}_{f}))\leqslant
\frac{1}{\log^n\frac{\varepsilon_0}{|x-x_0|}}\int\limits_{A(x_0,
|x-x_0|, \varepsilon_0)} \frac{Q(x)}{{|x-x_0|}^n} \,dm(x)\,.
\end{equation}
Observe that, the inequality~(\ref{eq3}) holds by the choice of $x$
and $\varepsilon_0.$ Then, by~(\ref{eq21}) and by
Proposition~\ref{pr1}, we obtain that
\begin{equation}\label{eq22A}
M(\Gamma(|D|, |K^{\,\prime}|, D^{\,\prime}_{f}))\leqslant
\frac{2}{\log^{n-1}\frac{\varepsilon_0}{|x-x_0|}}\frac{C\Omega_n2^n}{\log
2}\,.
\end{equation}
Combining~(\ref{eq16A}) and~(\ref{eq22A}), we obtain that
\begin{equation}\label{eq23A}
\frac{\omega_{n-1}}{A_0\left(\log\frac{2\lambda^2_n}{\delta\cdot
h(f(x), f(y))}\right)^{n-1}}\leqslant
\frac{2}{\log^{n-1}\frac{\varepsilon_0}{|x-x_0|}}\frac{C\Omega_n2^n}{\log
2}\,.
\end{equation}
Since~$\omega_{n-1}=n\Omega_n,$ the latter relation may be rewritten
in the following form:
\begin{equation}\label{eq24A}
\frac{1}{\log^{n-1}\frac{2\lambda^2_n}{\delta\cdot h(f(x),
f(y))}}\leqslant
\frac{1}{\log^{n-1}\frac{\varepsilon_0}{|x-x_0|}}\frac{A_0C2^{n+1}}{n\log
2}\,.
\end{equation}
By~(\ref{eq24A}), we obtain that
$$
\log\frac{2\lambda^2_n}{\delta\cdot h(f(x), f(y))}\geqslant
\log\frac{\varepsilon_0}{|x-x_0|}\cdot \left(\frac{n\log
2}{A_0C2^{n+1}}\right)^{\frac{1}{n-1}}\,,
$$
$$
\frac{2\lambda^2_n}{\delta\cdot h(f(x), f(y))}\geqslant
\left(\frac{\varepsilon_0}{|x-x_0|}\right)^{ \left(\frac{n\log
2}{A_0C2^{n+1}}\right)^{\frac{1}{n-1}}}\,,
$$
\begin{equation}\label{eq25A}
h(f(x), f(y))\leqslant \frac{2\lambda^2_n}{\delta
{\varepsilon_0}^{\left(\frac{n\log
2}{A_0C2^{n+1}}\right)^{\frac{1}{n-1}}}}|x-x_0|^{ \left(\frac{n\log
2}{A_0C2^{n+1}}\right)^{\frac{1}{n-1}}}\,.\end{equation}
By the definition of the chordal metric in~(\ref{eq3C}), since
$f(D)=D^{\,\prime}_{f}\subset B(0, R_0)$ it follows that
$$h(f(x),f(y))\geqslant \frac{|f(x)-f(y)|}{1+R^2_0}\,.$$
Thus, by~(\ref{eq25A}) we obtain that
\begin{equation}\label{eq26A}
|f(x)-f(y)|\leqslant \frac{2(1+R^2_0)\lambda^2_n}{\delta
{\varepsilon_0}^{\left(\frac{n\log
2}{A_0C2^{n+1}}\right)^{\frac{1}{n-1}}}}|x-x_0|^{ \left(\frac{n\log
2}{A_0C2^{n+1}}\right)^{\frac{1}{n-1}}}\,,\end{equation}
or, equivalently,
\begin{equation}\label{eq27A}
|f(x)-f(y)|\leqslant \frac{2(1+R^2_0)\lambda^2_n}{\delta
{\varepsilon_0}^{\beta_n}}|x-x_0|^{\beta_n}\,,
\end{equation}
because $\beta_n=\left(\frac{n\log
2}{A_0C2^{n+1}}\right)^{\frac{1}{n-1}}.$ It finishes the
proof.~$\Box$

\medskip
\begin{corollary}\label{cor2}
{\sl\, Under conditions of Theorem~\ref{th4}, the mapping $f$ has a
continuous extension at~$x_0$ and
$$|f(x)-f(x_0)|\leqslant \alpha\cdot|x-x_0|^{\beta_n}\,.$$
}
\end{corollary}
{\it Proof} of Corollary~\ref{cor2} is completely similar to the
proof of Corollary~\ref{cor1}.~$\Box$

\section{Proof of Theorem~\ref{th1}}

Since the set $B(y_0, r)\cap D$ is finitely connected for all
$0<r<r^{\,\prime}_0,$ the domain $D$ is finitely connected on its
boundary. Therefore, the domain $D$ is uniform (see
\cite[Theorem~3.2]{Na$_1$}). In other words, for every $r>0$ there
exists a number $\delta>0$ such that the inequality
\begin{equation}\label{eq17***}
M(\Gamma(F^{\,*},F, D))\geqslant \delta
\end{equation}
holds for all continua $F, F^*\subset D$ such that $h(F)\geqslant r$
and $h(F^{\,*})\geqslant r.$

\medskip Let us to prove that, for any $P\in E_D$
there exists $y_0\in \partial D$ such that $I(P)=\{y_0\}.$ We prove
this statement from the opposite, namely, suppose that there is a
$P\in E_D,$ which contains two points $x, y\in \partial D,$ $x\ne
y.$ In this case, there are at least two sequences $x_m, y_m\in
d_m,$ $m=1,2,\ldots ,$ which converge to $x$ and $y$ as
$m\rightarrow\infty,$ respectively (here $d_m$ denotes the
decreasing sequence of domains formed by some sequence of cuts
$\sigma_m,$ $m=1,2,\ldots ,$ corresponding to the prime end of $P$).
Let us join the points $x_m$ and $d_m$ with the path $\gamma_m$ in
the domain $d_m.$ Since $x\ne y,$ there exists $m_0\in {\Bbb N}$
such that $h(\gamma_m)\geqslant d(x,y)/2,$ $m>m_0.$ Choose any
nondegenerate continuum $C\subset D\setminus d_1.$ Then, due to the
uniformity of the domain $D$
\begin{equation}\label{eq1}
M(\Gamma(|\gamma_m|, C, D))\geqslant\delta_0>0
\end{equation}
for some $\delta_0>0$ and all $m>m_0.$ The relation~(\ref{eq1})
contradicts with the definition of the prime end~$P.$ Indeed, by the
definition of a cut~$\sigma_m,$ we obtain that $\Gamma(|\gamma_m|,
C, D)>\Gamma(\sigma_m, C, D).$ Now, due to~(\ref{eqSIMPLE}), we have
that
$$M(\Gamma(|\gamma_m|, C, D))\leqslant M(\Gamma(\sigma_m, C, D))\rightarrow 0$$
as $m\rightarrow\infty.$ The latter relation contradicts
with~(\ref{eq1}). Thus, $I(P)=\{y_0\}$ for some $y_0\in
\partial D.$

\medskip
It remains to prove the relation~(\ref{eq3A}). Let
$$\varepsilon_0=\min\{\varepsilon^{\,\prime}_0(x_0),
r_0^{\,\prime}, {\rm dist}\,(x_0, A), 1\}\,,$$
where $\varepsilon_0>0$ is a number from Proposition~\ref{pr1} with
$\varphi\equiv 1,$ besides that $r_0^{\,\prime}$ is a number from
conditions of the theorem, and $A$ is a continuum from the
definition of $\frak{F}^{A_0, R_0}_{Q, A, \delta}(D, P_0).$ By the
proving above, there is $x_0\in
\partial D$ such that $I(P_0)=x_0.$ It follows that, there is
$\rho_0=\rho_0(x_0)>0$ such that $B_{\rho}(P_0, \rho_0)\subset
B(x_0, \widetilde{\varepsilon_0}),$
$\widetilde{\varepsilon_0}=\varepsilon^2_0.$ It follows from the
definition of a regular domain that, $B_{\rho}(P_0, \rho)$ is
connected for sufficiently small $\rho.$ Thus, we may consider that
$B_{\rho}(P_0, \rho_0)$ is connected. Let $x, y\in B_{\rho}(P_0,
\rho_0)$ and let $f\in\frak{F}^{A_0, R_0}_{Q, A, \delta}(D, P_0).$
We may consider that $|x-x_0|\geqslant |y-x_0|.$ By the definition
of $r_0^{\,\prime},$ the points $x$ and $y$ may be joined by the
path $K,$ which is contained in the ball $\overline{B(x_0,
|x-x_0|)}.$
The further course of the proof is very similar to the proof of
Theorem~\ref{th3}. We may assume that, the path $K$ is Jordan.  Let
also $z, w\in f(A)\subset D^{\,\prime}_{f}$ and $u, v\in A$ be such
that
\begin{equation}\label{eq8B}
{\rm diam}\,f(A)=|z-w|=|f(u)-f(v)|\geqslant \delta\,.
\end{equation}
Since~$A$ is path connected, it is possible to join the points $u,
v$ by the path $K^{\,\prime}$ inside $A.$ Due to~(\ref{eq8B}), we
will to have that
\begin{equation}\label{eq9B}
{\rm diam}\,f(|K^{\,\prime}|)\geqslant |f(u)-f(v)|\geqslant
\delta\,.
\end{equation}
We may assume that the  $K^{\,\prime}$ is also Jordan.

Note that, $f(K)$ and $f(K^{\,\prime})$ are also Jordan ones, and
they do not split the space ${\Bbb R}^n$ (see the proof of this fact
in the course of proving Theorem~\ref{th3}). Therefore, $R=R(f(|K|),
f(|K^{\,\prime}|))$ is a ring domain. Let us denote
$\Gamma=\Gamma(f(|K|), f(|K^{\,\prime}|), {\Bbb R}^n).$ Then, by
Proposition~\ref{pr3}
\begin{equation}\label{eq10B}
M(\Gamma(f(|K|), f(|K^{\,\prime}|), \overline{{\Bbb
R}^n}))\geqslant\frac{\omega_{n-1}}{\left(
\log\frac{2\lambda^2_n}{h(f(|K|)))h(|f(K^{\,\prime}|))}\right)^{n-1}}\,.
\end{equation}
By~(\ref{eq9B}) and~(\ref{eq10B}), by the definition of $K,$
\begin{equation}\label{eq15B}
M(\Gamma(f(|K|), f(|K^{\,\prime}|), \overline{{\Bbb
R}^n}))\geqslant\frac{\omega_{n-1}}{\left(
\log\frac{2\lambda^2_n}{\delta\cdot h(f(x), f(y))}\right)^{n-1}}\,.
\end{equation}
Since $D^{\,\prime}_{f}$ is a $QED$-domain with a constant $A_0$
in~(\ref{eq4***}), by the condition~(\ref{eq15B}), we obtain that
\begin{equation}\label{eq16B}
M(\Gamma(f(|K|), f(|K^{\,\prime}|),
D^{\,\prime}_{f}))\geqslant\frac{\omega_{n-1}}{A_0\left(
\log\frac{2\lambda^2_n}{\delta\cdot h(f(x), f(y))}\right)^{n-1}}\,.
\end{equation}
Observe that,
\begin{equation}\label{eq17B}
\Gamma(|K|, |K^{\,\prime}|, D)>\Gamma(S(x_0, |x-x_0|), S(x_0,
\varepsilon_0), D)\,.
\end{equation}
The relation~(\ref{eq17B}) may be proved in the same way as the
relation~(\ref{eq17}) under the proof of Theorem~\ref{th3}. In this
case, by the minorizing property of the modulus (see, e.g.,
\cite[Theorem~6.4]{Va}), by~(\ref{eq17}) and~(\ref{eq3*!!}), we have
that
$$M(\Gamma(f(|K|), f(|K^{\,\prime}|), D^{\,\prime}_{f}))=
M(f(\Gamma(|K|, |K^{\,\prime}|, D)))\leqslant $$
\begin{equation}\label{eq20B}
\leqslant M(f(\Gamma(S(x_0, |x-x_0|), S(x_0, \varepsilon_0),
D)))\leqslant\int\limits_{A} Q(x)\cdot \eta^n(|x-x_0|)\ dm(x)\,,
\end{equation}
where $\eta$ is any nonnegative Lebesgue measurable function
satisfying the relation~(\ref{eq28*}) for $r_1=|x-x_0|,$
$r_2=\varepsilon_0.$ Put
$$\eta(t):=\begin{cases}\frac{1}
{\left(\log\frac{\varepsilon_0}{|x-x_0|}\right)
t}\,,& t\in(|x-x_0|, \varepsilon_0)\,, \\
0\,,& t\not\in(|x-x_0|, \varepsilon_0)\,.
\end{cases}$$
Observe that, the function $\eta$ satisfies the
relation~(\ref{eq28*}) for $r_1=|x-x_0|,$ $r_2=\varepsilon_0.$ Then,
by the condition~(\ref{eq20B}) we obtain that
\begin{equation}\label{eq21B}
M(\Gamma(f(|K|), f(|K^{\,\prime}|), D^{\,\prime}_{f}))\leqslant
\frac{1}{\log^n\frac{\varepsilon_0}{|x-x_0|}}\int\limits_{A(x_0,
|x-x_0|, \varepsilon_0)} \frac{Q(x)}{{|x-x_0|}^n} \,dm(x)\,.
\end{equation}
Since $|x-x_0|<\varepsilon^2_0,$ the relation~(\ref{eq3}) holds.
Now, by~(\ref{eq21B}) and due to Proposition~\ref{pr1}, we have that
\begin{equation}\label{eq22B}
M(\Gamma(f(|K|), f(|K^{\,\prime}|), D^{\,\prime}_{f}))\leqslant
\frac{2}{\log^{n-1}\frac{\varepsilon_0}{|x-x_0|}}\frac{C\Omega_n2^n}{\log
2}\,.
\end{equation}
Combining~(\ref{eq16B}) and~(\ref{eq22B}), we conclude that
\begin{equation}\label{eq23B}
\frac{\omega_{n-1}}{A_0\left( \log\frac{2\lambda^2_n}{\delta\cdot
h(f(x), f(y))}\right)^{n-1}}\leqslant
\frac{2}{\log^{n-1}\frac{\varepsilon_0}{|x-x_0|}}\frac{C\Omega_n2^n}{\log
2}\,.
\end{equation}
Next, we reason verbatim in the same way as when proving the
Theorem~\ref{th3} after the relation~(\ref{eq23}). As a result we
come to inequality
\begin{equation}\label{eq27B}
|f(x)-f(y)|\leqslant \frac{2(1+R^2_0)\lambda^2_n}{\delta
{\varepsilon_0}^{\beta_n}}|x-x_0|^{\beta_n}\,,
\end{equation}
because $\beta_n=\left(\frac{n\log
2}{A_0C2^{n+1}}\right)^{\frac{1}{n-1}}.$ The relation~(\ref{eq27B})
completes the proof.~$\Box$

\medskip
\begin{corollary}\label{cor3}
{\sl\, Under conditions of Theorem~\ref{th1}, $f$ has a continuous
extension to $P_0\in E_D,$ wherein
\begin{equation}\label{eq3B} |f(x)-f(P_0)|\leqslant
\alpha\cdot|x-x_0|^{\beta_n}\,,\end{equation}
where $\beta_n=\left(\frac{n\log
2}{A_0C2^{n+1}}\right)^{\frac{1}{n-1}}.$
  }
\end{corollary}

\medskip
\begin{proof}
By Theorem~\ref{th1}, $I(P)=x_0.$ From this it follows that, if
$x\rightarrow P_0$ and $y\rightarrow P_0,$ then $x\rightarrow x_0$
and $y\rightarrow x_0,$ as well. If the mapping $f$ did not have a
limit as $x\rightarrow P_0,$ then we may construct at least two
sequences $x_m \rightarrow P_0$ and $y_m\rightarrow P_0,$
$m\rightarrow\infty,$ such that $|f(x_m)-f(y_m)|\geqslant \delta>0$
for some positive $\delta>0$ and all $m=1,2,\ldots .$ But this
contradicts the inequality~(\ref{eq3A}). Therefore, the limit of $f$
as $x\rightarrow P_0$ exists. To prove~(\ref{eq24B}) it remains to
go in~(\ref{eq2.4.3}) to the limit as $y\rightarrow P_0.$~$\Box$
\end{proof}

\section{Proof of Theorem~\ref{th2}}

Due to the fact that the proof of this theorem is very similar to
all the previous ones, we will limit ourselves only to the scheme of
the proof. The fact that, for each $P\in E_D$ there exists $y_0\in
\partial D$ such that $I(P)=\{y_0\},$ may be established in the same
way as under proofs of Theorem~\ref{th1}. Let us to establish the
relation~(\ref{eq4A}). Let $\{x_0\}=I(P_0)$ and let
$$\varepsilon_0=\min\{\varepsilon^{\,\prime}_0(x_0),
r_0^{\,\prime}, {\rm dist}\,(x_0, A), 1\}\,,$$
where $\varepsilon_0>0$ is a number from Proposition~\ref{pr1} with
$\varphi\equiv 1,$ besides that, $r_0^{\,\prime}$ is a number from
conditions of the theorem, and $A$ is a continuum from the
definition of $\frak{F}^{A_0, R_0}_{Q, A, \delta}(D, P_0).$ It
follows that, there is $\rho_0=\rho_0(x_0)>0$ such that
$B_{\rho}(P_0, \rho_0)\subset B(x_0, \widetilde{\varepsilon_0}),$
$\widetilde{\varepsilon_0}=\varepsilon^2_0.$ It follows from the
definition of a regular domain that, $B_{\rho}(P_0, \rho)$ is
connected for sufficiently small $\rho.$ Let $x, y\in B_{\rho}(P_0,
\rho_0)$ and let $f\in\frak{R}^{A_0, R_0}_{Q, \delta}(D, P_0).$ We
may consider that $|x-x_0|\geqslant |y-x_0|.$ By the definition of
$r_0^{\,\prime},$ the points $x$ and $y$ may be joined by the path
$K,$ which is contained in the ball $\overline{B(x_0, |x-x_0|)}.$

Let $K_{f}\subset D^{\,\prime}_{f}$ be a path connected continuum
such that ${\rm diam\,} (K_{f})\geqslant \delta$ and such that
$h(f^{\,-1}(K_{f}),
\partial D)\geqslant \delta>0$ (such a continuum exists
by the definition of the class $\frak{R}^{A_0, R_0}_{Q, \delta}(D,
P_0)$).
Let also $z, w\in K_{f}\subset D^{\,\prime}_{f}$ be such that
\begin{equation}\label{eq8C}
{\rm diam}\,K_{f}=|z-w|\geqslant \delta\,.
\end{equation}
Since $K_{f}$ is path connected, we may join the points $z, w$ by a
Jordan path $K^{\,\prime}$ inside $K_{f}.$ Due to~(\ref{eq8C}), we
obtain that
%
$${\rm diam}\,f(|K^{\,\prime}|)\geqslant
|f(u)-f(v)|\geqslant \delta\,.$$
%
Next, we reason similarly to the proof of Theorem~\ref{th2} after
the relation~(\ref{eq9A}). We obtain that
\begin{equation}\label{eq27C}
|f(x)-f(y)|\leqslant \frac{2(1+R^2_0)\lambda^2_n}{\delta
{\varepsilon_0}^{\beta_n}}|x-x_0|^{\beta_n}\,,
\end{equation}
because $\beta_n=\left(\frac{n\log
2}{A_0C2^{n+1}}\right)^{\frac{1}{n-1}}.$ The relation~(\ref{eq27C})
contradicts completes the proof.~$\Box$

\medskip
\begin{corollary}\label{cor4}
{\sl\, Under conditions of Theorem~\ref{th2}, the mapping~$f$ has a
continuous extension to $P_0\in E_D,$ wherein
\begin{equation}\label{eq3BA} |f(x)-f(P_0)|\leqslant
\alpha\cdot|x-x_0|^{\beta_n}\,,\end{equation}
where $\beta_n=\left(\frac{n\log
2}{A_0C2^{n+1}}\right)^{\frac{1}{n-1}}.$
  }
\end{corollary}

\medskip
{\it Proof } is completely similar to proof of
Corollary~\ref{cor3}.~$\Box$

\medskip
\medskip
{\bf \noindent Evgeny Sevost'yanov} \\
{\bf 1.} Zhytomyr Ivan Franko State University,  \\
40 Bol'shaya Berdichevskaya Str., 10 008  Zhytomyr, UKRAINE \\
{\bf 2.} Institute of Applied Mathematics and Mechanics\\
of NAS of Ukraine, \\
1 Dobrovol'skogo Str., 84 100 Slavyansk,  UKRAINE\\
esevostyanov2009@gmail.com

\end{document}